\documentclass{amsart}
\usepackage{amssymb}
\usepackage{mathrsfs}
\usepackage{cases}
\usepackage{amsmath}
\usepackage{amsfonts}
\usepackage{pifont}
\usepackage{graphicx,amssymb,mathrsfs,amsmath}
\usepackage{tocvsec2}
\newtheorem{thm}{Theorem}[section]

\newtheorem{lem}[thm]{Lemma}

\theoremstyle{definition}
\newtheorem{defn}[thm]{Definition}
\newtheorem{example}[thm]{Example}
\theoremstyle{remark}
\newtheorem{rem}[thm]{Remark}
\numberwithin{equation}{section}
\begin{document}
\title[Entire and Analytical Solutions of Certain Classes...]{Entire and Analytical Solutions of Certain Classes of Abstract Degenerate Fractional Differential Equations and Their Systems$^{\ast}$}
%\author{Vladimir Fedorov}
%\address{Chelyabinsk State University, Chelyabinsk, 454 080, Russia}
%\email{kar@csu.ru}

\author{Marko Kosti\' c}
\address{Faculty of Technical Sciences,
University of Novi Sad,
Trg D. Obradovi\' ca 6, 21125 Novi Sad, Serbia}
\email{marco.s@verat.net}

%\author{Rodrigo Ponce}
%\address{Instituto de Matem\' atica y F\' ısica, Universidad de Talca, Casilla 747,
%Talca, Chile}
%\email{rodrigo.poncec@gmail.com}

{\renewcommand{\thefootnote}{} \footnote{
$^{\ast}$The text was submitted by the author for the English version of the journal.
\\ \text{  }  \ \    2010 {\it Mathematics
Subject Classification.} 47D06, 47D60,
47D62, 47D99.
\\ \text{  }  \ \    {\it Key words and phrases.} Abstract degenerate differential equations, Volterra integro-differential equations, fractional differential equations, entire and analytical solutions, well-posedness.
}}

\begin{abstract}
In this paper, we are primarily concerned with the study of entire and analytical solutions of abstract degenerate (multi-term) 
fractional differential equations with Caputo time-fractional derivatives. We also analyze systems of such equations and furnish several illustrative
examples to demonstrate usage of obtained theoretical results.
\end{abstract}
\maketitle

\section{INTRODUCTION AND PRELIMINARIES}

Fractional calculus has gained considerable popularity and importance during the past four decades, mainly due to its applications in diverse fields of science and engineering. 
Fairly complete information about fractional calculus and non-degenerate fractional
differential equations can be obtained
by consulting the references
\cite{bajlekova}, \cite{Diet}, \cite{kilbas}-\cite{knjigaho} and
\cite{Po}-\cite{samko}. 
Various types of abstract degenerate Volterra integro-differential equations and abstract degenerate (multi-term) fractional differential equations have been recently considered in \cite{fedorov}-\cite{vlad-mar} and \cite{filomat}-\cite{R-L-bilten} 
(cf. \cite{FK} for a comprehensive survey of results, as well as \cite{aliev1}, 
\cite{carol}, \cite{dem}, \cite{faviniyagi}, \cite{me152}, \cite{svir-fedorov}-\cite{svir3} and \cite{XL}-\cite{XL-HIGHER} for some other papers concerning the
abstract degenerate differential equations). 

It is well known that
the study of entire solutions of abstract differential equations was initiated by R. deLaubenfels \cite{l1-forum} in 1991. Concerning the theory
of abstract differential equations with integer order derivatives,
further contributions 
have been obtained by L. Autret \cite{auter0}, L. Autret-H. A. Emamirad \cite{auter}, T.-J. Xiao-J. Liang \cite{XL-entire}, Y. Mishura-Y. Tomilov \cite{mistom}, and the author \cite{knjigah}, \cite{sic}. The existence and uniqueness of entire and analytical solutions of the abstract non-degenerate time-fractional differential equations with Caputo derivatives have been investigated in \cite{fcaa}-\cite{systemss}. 
In a joint research paper with V. Fedorov \cite{vlad-mar-prim}, the author has recently considered a class of abstract degenerate multi-term fractional differential equations in locally convex spaces, pointing out that the methods proposed in \cite{XL-entire} (cf. also \cite[Section 4.4, pp. 167-175]{x263}), \cite{fcaa} and \cite[Remark 2.2(x)-(xi)]{vlad-mar-prim} can serve one to prove some results on the
existence and uniqueness of entire solutions of degenerate multi-term differential equations
with integer order derivatives (cf. \cite[Chapter 4]{svir-fedorov} for some basic results on the entire groups of solving operators for abstract degenerate differential equations of first order). Motivated primarily by this fact, in the second section of paper we consider the
existence and uniqueness of entire and analytical solutions to (systems) of degenerate multi-term fractional differential equations
with Caputo derivatives. It should also be noticed that in Subsection 2.1
we initiate the analysis of existence and uniqueness of entire and analytical solutions of some very important degenerate equations of mathematical physics in $L^{p}$ type spaces. 

We use the standard notation throughout the paper. 
Unless specifed otherwise,
we assume
that $X$ is a Hausdorff sequentially complete
locally convex space over the field of complex numbers. We use the shorthand SCLCS to denote such a space.
By
$L(X)$ we denote the space consisting of all continuous linear mappings from $X$ into
$X.$ By $\circledast$ we denote the fundamental systems of seminorms which defines the topology of $X.$
The Hausdorff locally convex topology on
$L(X)$ is defined in the usual way (see \cite[Section 1.1]{knjigaho}).
Let us recall that the space $L(X)$ is sequentially
complete provided that $X$ is barreled\index{barreled space} (\cite{meise}). 
If $A$ is a linear operator
acting on $X$,
then the domain, kernel space and range of $A$ will be denoted by
$D(A),$ $N(A)$ and $R(A),$
respectively. Since no confusion
seems likely, we will identify $A$ with its graph. The symbol $I$ stands for the identity operator on $X.$
If $C\in L(X)$ and $A$ is a closed linear operator acting on $X$, then we define the
$C$-resolvent set of $A,$
$\rho_{C}(A)$ for short, by $
\rho_{C}(A):=\{\lambda \in {\mathbb C}  \ | \ \lambda -A \mbox{ is
injective and } (\lambda-A)^{-1}C\in L(X)\};
$ $\rho(A)\equiv \rho_{I}(A).$ 
%Put $D_{\infty}(A):=\bigcap_{n\in {\mathbb N}}D(A^{n}).$
 If $V$ is a general topological vector space,
then a function $f :
\Omega \rightarrow V,$ where $\Omega$ is an open non-empty subset of ${\mathbb
C},$ is said to be analytic iff it is locally expressible in a
neighborhood of any point $z\in \Omega$ by a uniformly convergent
power series with coefficients in $V.$
We refer the reader to \cite{a43} and \cite[Section 1.1]{knjigaho} and references cited there for the basic information about vector-valued analytic functions. In our approach the space $X$ is sequentially complete, so that the analyticity of a mapping
$f: \Omega \rightarrow X$ is equivalent with its weak analyticity.

By ${\mathcal F}$ and ${\mathcal F}^{-1}$ we denote the Fourier transform on ${\mathbb R}^{n}$ and its inverse transform, respectively.
Given $\theta\in(0,\pi]$
in advance, define $\Sigma_{\theta}:=\{\lambda\in {\mathbb C}:\lambda\neq 0$, $|\arg(\lambda)|<\theta\}.$
Further on, $\lfloor\beta\rfloor:=\sup\{k\in {\mathbb Z}:\allowbreak k\leq\beta\},$ 
$\lceil\beta\rceil :=\inf\{k\in {\mathbb Z}:\beta\leq k\}$ ($\beta \in {\mathbb R}$), ${\mathbb N}_{n}:=\{1,\cdot \cdot \cdot,n\}$ and ${\mathbb N}_{n}^{0}:={\mathbb N}_{n} \cup \{0\}$ ($n\in {\mathbb N}$). 
The Gamma function is denoted by $\Gamma(\cdot)$ and the principal branch is always used to take
the powers; the convolution like
mapping $\ast$ is given by $f\ast g(t):=\int_{0}^{t}f(t-s)g(s)\,
ds .$ Set $g_{\zeta}(t):=t^{\zeta-1}/\Gamma(\zeta),$
$0^{\zeta}:=0$ ($\zeta>0,$ $t>0$), and
$g_{0}(t):=$ the Dirac $\delta$-distribution.
For a number $\zeta>0$ given in advance,
the Caputo fractional derivative ${\mathbf
D}_{t}^{\zeta}u$ (\cite{bajlekova}, \cite{knjigaho}) is defined
for those functions $u\in C^{\lceil \zeta \rceil-1}([0,\infty) : X)$ for which
$g_{\lceil \zeta \rceil-\zeta} \ast (u-\sum_{j=0}^{\lceil \zeta \rceil-1}u^{(j)}(0)g_{j+1}) \in
C^{\lceil \zeta \rceil}([0,\infty) : X),$ by
$$
{\mathbf
D}_{t}^{\zeta}u(t):=\frac{d^{\lceil \zeta \rceil}}{dt^{\lceil \zeta \rceil}}\Biggl[g_{\lceil \zeta \rceil-\zeta}
\ast \Biggl(u-\sum_{j=0}^{\lceil \zeta \rceil-1}u^{(j)}(0)g_{j+1}\Biggr)\Biggr].
$$
The
Mittag-Leffler function $E_{\beta,\gamma}(z)$ ($\beta>0,$ $\gamma \in {\mathbb R}$) is defined by
$$
E_{\beta,\gamma}(z):=\sum_{k=0}^{\infty}\frac{z^{k}}{\Gamma(\beta
k+\gamma)},\quad z\in {\mathbb C}.
$$
In this place, we assume that
$1/\Gamma(\beta k+\gamma)=0$ if $\beta k+\gamma \in -{{\mathbb
N}_{0}}.$ Set, for short, $E_{\beta}(z):=E_{\beta,1}(z),$ $z\in
{\mathbb C}.$
The asymptotic behaviour of the entire function $E_{\beta,\gamma}(z)$
is given in the following auxiliary lemma (see e.g. \cite[Section 1.3]{knjigaho}): 

\begin{lem}\label{1.1}
Let $0<\sigma<\frac{1}{2}\pi .$ Then, for every $z\in {\mathbb C}
\setminus \{0\}$ and $l\in {\mathbb N} \setminus \{1\},$ 
$$
E_{\beta,\gamma}(z)=\frac{1}{\beta}\sum
\limits_{s}Z_{s}^{1-\gamma}e^{Z_{s}}-\sum
\limits^{l-1}_{j=1}\frac{z^{-j}}{\Gamma(\gamma-\beta
j)}+O\bigl(|z|^{-l}\bigr),\quad |z|\rightarrow \infty ,
$$
where $Z_{s}$ is defined by $Z_{s}:=z^{1/\beta}e^{2\pi i s/\beta}$
and the first summation is taken over all those integers $s$
satisfying $|\arg (z) + 2\pi s|<\beta(\frac{\pi}{2}+\sigma).$
\end{lem}
For further information about the Mittag-Leffler functions and the abstract Volterra integro-differential equations in Banach and locally convex spaces, the reader may consult \cite{bajlekova}, \cite{prus}, \cite{knjigaho} and references cited there.

Assume 
that
$n\in {\mathbb N}$ and $iA_{j},\ 1\leq j\leq n$ are commuting
generators of bounded $C_{0}$-groups on a Banach space $X.$ Set $A:=(A_{1},\cdot \cdot \cdot,A_{n})$ and
$A^{\eta}:=A_{1}^{\eta_{1}}\cdot \cdot \cdot A_{n}^{\eta_{n}}$ for
any $\eta=(\eta_{1},\cdot \cdot \cdot, \eta_{n})\in {{\mathbb
N}_{0}^{n}};$ 
denote by ${\mathcal D}({{\mathbb R}^{n}})$
and ${\mathcal S}({{\mathbb R}^{n}})$ the Schwartz space of $C^{\infty}({\mathbb R}^{n})$-functions with compact support and 
the Schwartz space of rapidly decreasing functions on ${{\mathbb R}^{n}},$ respectively (in the sequel, the meaning of symbol $A$ will be clear from the context). Let $k=1+\lfloor
n/2\rfloor.$ For every $\xi=(\xi_{1},\cdot \cdot \cdot, \xi_{n}) \in
{{\mathbb R}^{n}}$ and $u\in {\mathcal F}L^{1}({\mathbb R}^{n})= \{
{\mathcal F}f : f \in L^{1}({{\mathbb R}^{n}}) \},$ we set
$|\xi|:=(\sum_{j=1}^{n}\xi_{j}^{2})^{1/2},$
$(\xi,A):=\sum_{j=1}^{n}\xi_{j}A_{j}$ and
$
u(A)x:=\int_{{\mathbb R}^{n}}{\mathcal
F}^{-1}u(\xi)e^{-i(\xi,A)}x\, d\xi,\ x\in X.
$
Then $u(A)\in
L(X),$ $u\in {\mathcal F}L^{1}({{\mathbb R}^{n}})$ and there exists a finite constant
$M\geq 1$ such that 
$
\|u(A)\|\leq M \|{\mathcal
F}^{-1}u\|_{L^{1}({{\mathbb R}^{n}})},\ u\in {\mathcal
F}L^{1}({\mathbb R}^{n}).
$
Let $N\in {\mathbb N},$ and let
$p(x)=\sum_{|\eta|\leq N}a_{\eta}x^{\eta},$ $x\in {\mathbb R}^{n}$ be a complex polynomial. Then we define
$
p(A):=\sum_{|\eta|\leq N}a_{\eta}A^{\eta}\mbox{ and }
X_{0}:=\bigl\{\phi(A)x : \phi \in {\mathcal S}({{\mathbb R}^{n}}),\ x\in
X\bigr\}.
$
We know that the operator $p(A)$ is
closable and that the following holds:
\begin{itemize}
\item[($\triangleright$)]
$\overline{X_{0}}=X,$ $X_{0}\subseteq \bigcap _{\eta \in {{\mathbb
N}_{0}^{n}}}D(A^{\eta}),$ $\overline{p(A)_{|X_{0}}}=\overline{p(A)}$
and\\ $\phi(A)p(A)\subseteq p(A)\phi(A)=(\phi p)(A),$ $\phi \in
{\mathcal S}({{\mathbb R}^{n}}).$
\end{itemize}
Denote by ${\mathbb C}^{m,m}$ the ring of $m \times m$ matrices over ${\mathbb C}$; $I_{m}$ stands for the identity matrix of format $m\times m$ ($m\in {\mathbb N}$). If 
$P(x)=[p_{ij}(x)]$ is an $m\times m$ matrix of polynomials of $x\in {\mathbb R}^{n},$ then 
there exist $d\in {\mathbb N}$ and matrices $P_{\eta}\in {\mathbb C}^{m,m}$ such that $P(x)=\sum_{|\eta|\leq d}P_{\eta}x^{\eta},$ $x\in {\mathbb R}^{n}.$ Then the operator $P(A):=\sum_{|\eta|\leq d}P_{\eta}A^{\eta}
$ is closable on $X^{m}.$
For further information concerning the functional calculus for commuting generators of bounded $C_{0}$-groups, see \cite{l1}, \cite{knjigaho} and \cite{zheng-pacific}-\cite{quan-miao}.

The proof of following auxiliary lemma, which is probably known in the existing literature, is included for the sake of completeness.

\begin{lem}\label{lap-inj}
Suppose that $1\leq p<\infty,$ $n\in {\mathbb N}$ and $X:=L^{p}({\mathbb R}^{n}).$ Denote by $\Delta_{p,n}$ the operator $\Delta$ acting on $X$ with its maximal distributional domain. Then  $\Delta_{p,n}$ is injective.
\end{lem}

\begin{proof}
If $1<p<\infty,$ then the statement immediately follows from the fact that the operator $-\Delta_{p,n}$ is non-negative, with dense domain and range (cf. \cite[pp. 256, 266]{MSP}). Suppose now that $p=1$ and $\Delta_{p,n}f=0$ for some $f\in X.$ Then \cite[Lemma 3.2]{MSP} implies that, for every $\varphi \in {\mathcal D}({\mathbb R}^{n})$ and for every multi-index $\eta \in {\mathbb N}_{0}^{n},$ the function 
$\varphi \ast f$ belongs to the space ${\mathcal T}$ consisting of those $C^{\infty}({\mathbb R}^{n})$-functions whose any partial
derivative belongs to $L^{1}({\mathbb R}^{n}) \cap L^{\infty}({\mathbb R}^{n}).$ Since $\Delta_{p,n}(\varphi \ast f)=\varphi \ast \Delta_{p,n}f=0,$ $\varphi \in {\mathcal D}({\mathbb R}^{n})$ and the operator $\Delta_{{\mathcal T}}$ is injective by \cite[Remark 3.3]{MSP}, we have that $\varphi \ast f=0,$ $\varphi \in  {\mathcal D}({\mathbb R}^{n}).$ Hence, $f=0.$
\end{proof}

\section{FORMULATION AND PROOF OF MAIN RESULTS. EXAMPLES AND APPLICATIONS}

Before stating our first main result, we need to repeat some notations and preliminaries from \cite{vlad-mar-prim}.
Suppose that $n\in {\mathbb N},$ $0<\zeta \leq 2,$
$q_{0},q_{1},\cdot \cdot \cdot ,q_{n}$ are given non-negative integers satisfying $q_{0}=0$ and $0<q_{1}\leq q_{2}\leq \cdot \cdot \cdot \leq q_{n}.$ Let $A_{0},A_{1},\cdot \cdot \cdot, A_{n-1},A_n$ be closed linear operators acting on $X.$ Set $A_{n}:=B,$ $T_{i}u(t):=A_{i}({\mathbf D}_{t}^{\zeta})^{q_{i}}u(t),$ $t\geq 0,$ $i\in {\mathbb N}_{n}^{0}$ and
$$
P_{\lambda}:=\lambda^{q_{n}\zeta}B+\sum \limits_{i=0}^{n-1}\lambda^{q_{i}\zeta}A_{i},\quad \lambda \in {\mathbb C} \setminus \{0\}.
$$
Of concern is the following abstract degenerate multi-term Cauchy problem:
\begin{equation}\label{snarky}
\sum \limits_{i=0}^{n}T_{i}u(t)=0,\quad t\geq 0,
\end{equation}
accompanied with the following initial conditions:
\begin{align}
\notag
& \Bigl(\bigl ({\mathbf D}_{t}^{\zeta}\bigr)^{j}u(t)\Bigr)_{t=0}=u_{j},\ j\in {\mathbb N}_{q_{n}-1}^{0},\mbox{ if }\zeta \in (0, 1],\mbox{ resp., }\\\label{puppy}&
\Bigl(\bigl ({\mathbf D}_{t}^{\zeta}\bigr)^{j}u(t)\Bigr)_{t=0}=u_{j},\ j\in {\mathbb N}_{q_{n}-1}^{0}; \ \Bigl(\frac{d}{dt}\bigl ({\mathbf D}_{t}^{\zeta}\bigr)^{j}u(t)\Bigr)_{t=0}=v_{j},\ j\in {\mathbb N}_{q_{n}-1}^{0},\mbox{ if }\zeta \in (1,2].
\end{align}
In \cite{vlad-mar-prim}, we have considered the abstract Cauchy problem [(\ref{snarky})-(\ref{puppy})]
with $0<\zeta \leq 1.$ The notion of a strong solution of
problem [(\ref{snarky})-(\ref{puppy})], introduced in the first part of following definition, coincides with the corresponding notion introduced in \cite[Definition 1]{vlad-mar-prim} in the case that $0<\zeta \leq 1.$

\begin{defn}\label{RES}
\begin{itemize}
\item[(i)] A function $u\in C([0,\infty): X)$ is said to be a strong solution of
problem [(\ref{snarky})-(\ref{puppy})]
iff the term $T_{i}u(t)$ is well defined and continuous for any $t\geq 0,$ $i\in {\mathbb N}_{n}^{0}$,
and [(\ref{snarky})-(\ref{puppy})] holds identically on $[0,\infty).$ 
\item[(ii)] A function $u\in C([0,\infty): X)$ is said to be an entire solution of
problem [(\ref{snarky})-(\ref{puppy})] iff $u(\cdot)$ is a strong solution of [(\ref{snarky})-(\ref{puppy})] and it can be analytically extended to the whole complex plane, as well as any of the terms $A_{i}u^{(p)}(\cdot)$ ($0\leq i\leq n,$ $p\in {\mathbb N}_{0}$) can be analytically extended to the whole complex plane.
\item[(iii)] A function $u\in C([0,\infty): X)$ is said to be an analytical solution of
problem [(\ref{snarky})-(\ref{puppy})] on the region ${\mathbb C} \setminus (-\infty,0]$ iff $u(\cdot)$ is a strong solution of [(\ref{snarky})-(\ref{puppy})] and it can be extended to the whole complex plane, analytically on the region ${\mathbb C} \setminus (-\infty,0]$ and continuously on the region ${\mathbb C} \setminus (-\infty,0)$, as well as any of the terms $A_{i}({\mathbf D}_{t}^{\zeta})^{p}u(t)$ ($0\leq i\leq n,$  $p\in {\mathbb N}_{0},$ $t\geq 0$) is well defined and can be extended to the whole complex plane, analytically on the region ${\mathbb C} \setminus (-\infty,0]$ and continuously on the region ${\mathbb C} \setminus (-\infty,0).$ 
\end{itemize}
\end{defn}

Set $S_{\omega}:=\{ j\in {{\mathbb N}_{n}^{0}} : q_{j}-1\geq \omega\}$ ($\omega \in {{\mathbb N}_{q_{n}-1}^{0}}$). We need to introduce the following condition:
\begin{equation}\label{818}
-\infty <\nu ' <\min \limits_{\omega \in {{\mathbb N}_{q_{n}-1}^{0}}} \Bigl[ -\Bigl(q_{n}-1-\omega +\max \bigl\{q_{j} : j\in {\mathbb N}_{n}^{0} \setminus S_{\omega}\bigr\} \Bigr) \Bigr].
\end{equation}
Then $n\in S_{\omega}$ for all $\omega \in {{\mathbb N}_{q_{n}-1}^{0}},$ and  (\ref{818}) holds provided that $-\infty <\nu'<1-q_{n}.$ 

Now we are ready to formulate the following theorem.

\begin{thm}\label{GZA}
Suppose that the operator $C\in L(X)$ is injective, $CA_{i}\subseteq A_{i}C,$ $i\in {\mathbb N}_{n}^{0},$
$0<\zeta \leq 2,$
$\phi \in (-\pi,\pi],$ $\theta \in (\pi-\pi \zeta,\pi-(\pi \zeta)/2),$ $a>r>0$ and $\nu'$ satisfies (\ref{818}). Assume, further, that the following holds:
\begin{itemize}
\item[(i)] The operator family $\{(1+|\lambda|)^{-\nu'}P_{\lambda^{1/\zeta}}^{-1}C : \lambda \in e^{i\phi}\Sigma_{(\zeta \pi/2)+\theta},\ |\lambda|\geq r \}\subseteq L(X)$ is both equicontinuous and strongly continuous.
\item[(ii)] For every $x\in X$ and $i\in {\mathbb N}_{n-1}^{0},$ the mapping $\lambda \mapsto A_{i}P_{\lambda^{1/\zeta}}^{-1}Cx,$ $\lambda \in e^{i\phi}\Sigma_{(\zeta \pi/2)+\theta},\ |\lambda|\geq r$ is continuous and there exists $v_{i}\in {\mathbb N}$ such that the operator family $\{(1+|\lambda|)^{-v_{i}}A_{i}P_{\lambda^{1/\zeta}}^{-1}C : \lambda \in e^{i\phi}\Sigma_{(\zeta \pi/2)+\theta},\ |\lambda|\geq r\}\subseteq L(X)$ is equicontinuous.
\end{itemize}
Denote by ${\mathfrak W}$ (${\mathfrak W}_{e}$)
the subspace of $X^{q_{n}},$ resp. $X^{2q_{n}}$, consisting of all initial values $(u_{0},\cdot \cdot \cdot,u_{q_{n}-1})\in X^{q_{n}},$
resp. $(u_{0},\cdot \cdot \cdot,u_{q_{n}-1}; v_{0},\cdot \cdot \cdot,v_{q_{n}-1})\in X^{2q_{n}},$ subjected to some analytical solution $u(\cdot)$ of problem [(\ref{snarky})] on the region ${\mathbb C} \setminus (-\infty,0]$  (entire solution $ u(\cdot)$ of problem [(\ref{snarky})]). Then ${\mathfrak W}$
is dense in 
$(C(
\bigcap_{j=0}^{n}D(A_{j})))^{q_{n}}$ for the topology of $X^{q_{n}},$ provided that $0<\zeta<1,$ resp. $(C(
\bigcap_{j=0}^{n}D(A_{j})))^{2q_{n}}$ for the topology of $X^{2q_{n}},$ provided that $1<\zeta<2$; furthermore, ${\mathfrak W}_{e}$
is dense in 
$(C(
\bigcap_{j=0}^{n}D(A_{j})))^{q_{n}}$  for the topology of $X^{q_{n}},$ provided that $\zeta=1,$ resp.  $(C(
\bigcap_{j=0}^{n}D(A_{j})))^{2q_{n}}$  for the topology of $X^{2q_{n}},$ provided that $\zeta=2.$
\end{thm}

To prove Theorem \ref{GZA}, we need the following lemma (cf. also \cite[Lemma 1.1, Theorem 1.1]{svir3}).

\begin{lem}\label{tuga-jesenja}
Let $x\in X.$ Then the mapping 
$$
\lambda \mapsto P_{(\lambda e^{i\phi})^{1/\zeta}}^{-1}Cx,\quad \lambda \in \Sigma_{(\zeta \pi/2)+\theta},\ |\lambda|> r
$$ 
is analytic.
\end{lem}

\begin{proof}
Without loss of generality, we may assume that $q_{i}=i$ ($i\in {\mathbb N}_{n}^{0}$), $\zeta=1$ and $\phi=0.$ 
Clearly, (ii) holds 
for every $x\in X$ and $i\in {\mathbb N}_{n}^{0}.$  Furthermore, the following analogon of the Hilbert resolvent equation holds:
\begin{align*}
& P_{\lambda}^{-1}C^{2}x  -P_{z}^{-1}C^{2}x
=(z-\lambda)P_{\lambda}^{-1}C\\& \times \Biggl[ \sum_{k=1}^{n-1}{n \choose k}\bigl(z-\lambda \bigr)^{k-1}\lambda^{n-k}B+\sum_{k=1}^{n-2}{n-1 \choose k}\bigl(z-\lambda \bigr)^{k-1}\lambda^{n-1-k}A_{n-1}+\cdot \cdot \cdot +A_{1} \Biggr]
\\ & \times P_{z}^{-1}Cx,\mbox{ provided } \lambda,\ z\in  \Sigma_{(\zeta \pi/2)+\theta}\mbox{ and } |\lambda|,\ |z|> r.
\end{align*}
This implies that the mapping $\lambda \mapsto P_{\lambda}^{-1}C^{2}x,$ $\lambda \in \Sigma_{(\zeta \pi/2)+\theta},$ $|\lambda|>r$
is weakly analytic and therefore analytic, as well as that
\begin{align*}
&\frac{d}{d\lambda}\Bigl \langle x^{\ast},  P_{\lambda}^{-1}C^{2}x \Bigr \rangle 
\\ & =-\Bigl \langle x^{\ast} , P_{\lambda}^{-1}\bigl[ n \lambda^{n-1}B+(n-1)\lambda^{n-2}A_{n-1}+\cdot \cdot \cdot +A_{1}\bigr]P_{\lambda}^{-1}Cx\Bigr \rangle,
\end{align*}
provided $x^{\ast} \in X^{\ast}, $ $\ \lambda 
 \in \Sigma_{(\zeta \pi/2)+\theta}$ and $|\lambda|>r.$ Using the Morera theorem and the observation from \cite[Remark 2.7]{ralf}, the above implies that the mapping $ \lambda \mapsto P_{\lambda}^{-1}Cx,$ $\lambda \in \Sigma_{(\zeta \pi/2)+\theta},\ |\lambda|> r $ is analytic, as claimed.
\end{proof}

Now we can proceed to the proof of Theorem \ref{GZA}.

\begin{proof}
Suppose first $0<\zeta\leq 1.$ Clearly, $(\zeta \pi/2)+\theta<\pi,$ $\pi \zeta /2 >\pi-(\zeta \pi/2)-\theta$ and we can find a number $b\in {\mathbb R}$ satisfying
$$
1<b<\frac{\pi \zeta /2}{\pi -(\zeta \pi/2)-\theta}.
$$
Denote by $\Gamma$ the upwards oriented boundary of the region $\{\lambda \in \Sigma_{(\zeta \pi/2)+\theta} : |\lambda|\geq r\}.$ Let $\Omega$ be the open region on the left of $\Gamma.$ Then there exists a sufficiently large number $R>0$ such that
$a-\lambda \in \Sigma_{\pi-(\zeta \pi/2)-\theta}$ for all $\lambda \in \Omega \cup \Gamma$ with $|\lambda| \geq R.$
This implies $|e^{-\epsilon (a-\lambda)^{b/\zeta}}|=e^{-\epsilon \Re ((a-\lambda)^{b/\zeta})}\leq
e^{-\epsilon |a-\lambda|^{b/\zeta}\cos(b\zeta^{-1}(\pi-(\pi \zeta/2)-\theta))}
,$ provided $\epsilon>0,$ $\lambda \in \Omega \cup \Gamma$ and $|\lambda|\geq R.$
Keeping in mind Lemma \ref{1.1}, we obtain the existence of a constant $c_{\zeta}'>0$ such that $|E_{\zeta}( z^{\zeta}\lambda e^{i\phi})|\leq E_{\zeta}(|z|^{\zeta} |\lambda|)\leq c_{\zeta}'e^{|z||\lambda|^{1/\zeta}}$ for all $z\in {\mathbb C}$ and $\lambda \in {\mathbb C}.$ Hence, there exists a constant $c_{\zeta}>0$ such that
\begin{equation}\label{estimate}
\Bigl| e^{-\epsilon (a-\lambda)^{b/\zeta}}E_{\zeta}\bigl( z^{\zeta}\lambda e^{i\phi}\bigr) \Bigr| \leq c_{\zeta}e^{-\epsilon |a-\lambda|^{b/\zeta}\cos(b\zeta^{-1}(\pi-(\pi \zeta/2)-\theta))+|z||\lambda|^{1/\zeta}},
\end{equation}
for any $z\in {\mathbb C},\ \epsilon>0$ and $\lambda \in \Omega.$
Suppose now that $x_{w}\in \bigcap_{j=0}^{n}D(A_{j})$ for all $w\in {\mathbb N}_{q_{n}-1}^{0}.$ Then (i) and the estimate (\ref{estimate}) enable us to define the function $z\mapsto u_{\epsilon}(z),$ $z\in {\mathbb C},$ for any $\epsilon>0,$ by
$$
u_{\epsilon}(z):=\sum \limits_{w=0}^{q_{n}-1}\sum \limits_{j\in S_{\omega}}\frac{1}{2\pi i}\int_{\Gamma}e^{-\epsilon (a-\lambda)^{b/\zeta}}E_{\zeta}\bigl( z^{\zeta}\lambda e^{i\phi}\bigr) \bigl(\lambda e^{i\phi}\bigr)^{q_{j}-1-w}P_{(\lambda e^{i\phi})^{1/\zeta}}^{-1}CA_{j}x_{w}\, d\lambda .
$$
It can be simply verified that the mapping $z\mapsto u_{\epsilon}(z),$ $z\in {\mathbb C} \setminus (-\infty,0)$ is continuous ($\epsilon>0$).
Using Lemma \ref{1.1} and the proof of \cite[Theorem 2]{vlad-mar-prim}, it readily follows that the mapping $z\mapsto u_{\epsilon}(z),$ $z\in  {\mathbb C} \setminus (-\infty,0]$ is analytic ($\epsilon>0$), with
\begin{equation}\label{nujabes}
u^{\prime}_{\epsilon}(z)=\sum \limits_{w=0}^{q_{n}-1}\sum \limits_{j\in S_{\omega}}\frac{1}{2\pi i}\int_{\Gamma}e^{-\epsilon (a-\lambda)^{b/\zeta}}z^{\zeta-1}E_{\zeta,\zeta}\bigl( z^{\zeta}\lambda e^{i\phi}\bigr) \bigl(\lambda e^{i\phi}\bigr)^{q_{j}-w}P_{(\lambda e^{i\phi})^{1/\zeta}}^{-1}CA_{j}x_{w}\, d\lambda ,
\end{equation}
for any $\epsilon>0$ and $z\in {\mathbb C} \setminus (-\infty,0];$ furthermore, the mapping $z\mapsto u_{\epsilon}(z),$ $z\in {\mathbb C}$ is entire provided $\epsilon>0,$ $\zeta=1$ and, in this case, (\ref{nujabes}) holds 
for any $\epsilon>0$ and $z\in {\mathbb C}.$ The proof of \cite[Theorem 2]{vlad-mar-prim} also shows that the term $({\mathbf D}_{t}^{\zeta})^{p}u_{\epsilon}(t),$ $t\geq 0$ is well defined for any $p\in {\mathbb N}_{0}$ and  $\epsilon>0,$ with
\begin{align}
\notag 
({\mathbf D}_{t}^{\zeta})^{p}& u_{\epsilon}(t)
\\\label{london-acid} &=\sum \limits_{w=0}^{q_{n}-1}\sum \limits_{j\in S_{\omega}}\frac{1}{2\pi i}\int_{\Gamma}e^{-\epsilon (a-\lambda)^{b/\zeta}}E_{\zeta}\bigl( t^{\zeta}\lambda e^{i\phi}\bigr) \bigl(\lambda e^{i\phi}\bigr)^{p+q_{j}-1-w}P_{(\lambda e^{i\phi})^{1/\zeta}}^{-1}CA_{j}x_{w}\, d\lambda ;
\end{align}
cf. also the formula \cite[(1.25)]{bajlekova}. Combined with the Cauchy theorem, (ii) and Lemma \ref{1.1},
the above implies that the term 
$A_{i}({\mathbf D}_{t}^{\zeta})^{p}u_{\epsilon}(t)$ is well defined for $t\geq 0,$ $i\in {\mathbb N}_{n}^{0},$ $p\in {\mathbb N}_{0}$ and $\epsilon>0,$
with
\begin{align*}
& A_{i}\bigl({\mathbf D}_{t}^{\zeta}\bigr)^{p}u_{\epsilon}(t)
\\& =\sum \limits_{w=0}^{q_{n}-1}\sum \limits_{j\in S_{\omega}}\frac{1}{2\pi i}\int_{\Gamma}e^{-\epsilon (a-\lambda)^{b/\zeta}}E_{\zeta}\bigl( t^{\zeta}\lambda e^{i\phi}\bigr) \bigl(\lambda e^{i\phi}\bigr)^{p+q_{j}-1-w}A_{i}P_{(\lambda e^{i\phi})^{1/\zeta}}^{-1}CA_{j}x_{w}\, d\lambda.
\end{align*}
This implies that, for every $\epsilon>0,$ any of the terms $A_{i}({\mathbf D}_{t}^{\zeta})^{p}u_{\epsilon}(\cdot)$ ($0\leq i\leq n,$ $p\in {\mathbb N}_{0}$) can be extended to the whole complex plane, analytically on the region ${\mathbb C} \setminus (-\infty,0]$ and continuously on the region ${\mathbb C} \setminus (-\infty,0);$ we only need to replace the variable $t\geq 0,$ appearing in the above formula, with the variable $z\in {\mathbb C}.$ 
Furthermore,
\begin{align*}
& \sum \limits_{i=0}^{n}A_{i}({\mathbf D}_{t}^{\zeta})^{q_{i}}u_{\epsilon}(t)
\\ & =\sum \limits_{w=0}^{q_{n}-1}\sum \limits_{j\in S_{\omega}}\sum \limits_{i=0}^{n}\frac{1}{2\pi i}\int_{\Gamma}e^{-\epsilon (a-\lambda)^{b/\zeta}}E_{\zeta}\bigl( t^{\zeta}\lambda e^{i\phi}\bigr) \bigl(\lambda e^{i\phi}\bigr)^{q_{i}+q_{j}-1-w}A_{i}P_{(\lambda e^{i\phi})^{1/\zeta}}^{-1}CA_{j}x_{w}\, d\lambda
\\ & =\sum \limits_{w=0}^{q_{n}-1}\sum \limits_{j\in S_{\omega}} \frac{1}{2\pi i}\int_{\Gamma}e^{-\epsilon (a-\lambda)^{b/\zeta}}E_{\zeta}\bigl( t^{\zeta}\lambda e^{i\phi}\bigr)\bigl(\lambda e^{i\phi}\bigr)^{q_{j}-1-w}CA_{j}x_{\omega}\, d\lambda=0,\ t\geq 0,\ \epsilon>0,
\end{align*} 
so that for each $\epsilon>0$ the mapping $t\mapsto u_{\epsilon}(t),$ $t\geq 0$ 
is an analytical solution of problem (\ref{snarky}) on the region ${\mathbb C} \setminus (-\infty,0]$ (entire solution of problem (\ref{snarky}), provided that $\zeta=1$). Let $u_{l}^{\epsilon}=( ({\mathbf D}_{t}^{\zeta})^{l}u_{\epsilon}(t))_{t=0},$ $l \in {\mathbb N}_{q_{n}-1}^{0}$ ($\epsilon>0$).
Now we will prove that $(u_{0}^{\epsilon},\cdot \cdot \cdot,u_{q_{n}-1}^{\epsilon})$ converges to $e^{-i\phi}(Cx_{0},\cdot \cdot \cdot , Cx_{q_{n}-1})$ as $\epsilon \rightarrow 0+,$ for the topology of $X^{q_{n}}$ (cf. also \cite[Remark 1(x)]{vlad-mar-prim}). Let $\omega \in {\mathbb N}_{q_{n}-1}^{0}$ and $l\in {\mathbb N}_{q_{n}-1}^{0}$ be fixed. Keeping in mind (\ref{london-acid}), it suffices to prove that the following holds:
$$
\lim \limits_{\epsilon \rightarrow 0+}\sum \limits_{j\in S_{\omega}}\frac{1}{2\pi i}\int_{\Gamma}e^{-\epsilon (a-\lambda)^{b/\zeta}}\bigl(\lambda e^{i\phi}\bigr)^{l+q_{j}-1-w}P_{(\lambda e^{i\phi})^{1/\zeta}}^{-1}CA_{j}x_{w}\, d\lambda =e^{-i\phi}\delta_{ \omega l}Cx_{\omega},
$$
i.e., that
\begin{align}
\notag \lim \limits_{\epsilon \rightarrow 0+}& \frac{1}{2\pi i}\int_{\Gamma}e^{-\epsilon (a-\lambda)^{b/\zeta}}\bigl(\lambda e^{i\phi}\bigr)^{l-1-w}
\\\label{AJ} & \times \Biggl[ Cx_{\omega}-\sum \limits_{j\in {\mathbb N}_{n}^{0} \setminus S_{\omega}}\bigl(\lambda e^{i\phi}\bigr)^{q_{j}}P_{(\lambda e^{i\phi})^{1/\zeta}}^{-1}CA_{j}x_{w}\Biggr]\, d\lambda =e^{-i\phi}\delta_{ \omega l}Cx_{\omega},
\end{align}
where $\delta_{ \omega l}$ denotes the Kronecker delta. Since $|e^{-\epsilon (a-\lambda)^{b/\zeta}}|\leq 1,$ $\lambda \in \Gamma ,$ $\epsilon>0,$ (\ref{818}) and (i) holds, we have that there exists $\sigma>0$ such that
$$
\Biggl |  e^{-\epsilon (a-\lambda)^{b/\zeta}}\bigl(\lambda e^{i\phi}\bigr)^{l-1-w}\bigl(\lambda e^{i\phi}\bigr)^{q_{j}}P_{(\lambda e^{i\phi})^{1/\zeta}}^{-1}CA_{j}x_{w} \Biggr |\leq \mbox{Const. } |\lambda|^{-1-\sigma},
$$  
for any $\lambda \in \Gamma ,$ $\epsilon>0$ and $j\in {\mathbb N}_{n}^{0} \setminus S_{\omega}.$ Applying the dominated convergence theorem, Lemma \ref{tuga-jesenja} and the Cauchy theorem, we get that
\begin{align*}
\lim \limits_{\epsilon \rightarrow 0+}&\frac{1}{2\pi i}\int_{\Gamma}e^{-\epsilon (a-\lambda)^{b/\zeta}}\bigl(\lambda e^{i\phi}\bigr)^{l-1-w}\sum \limits_{j\in {\mathbb N}_{n}^{0} \setminus S_{\omega}}\bigl(\lambda e^{i\phi}\bigr)^{q_{j}}P_{(\lambda e^{i\phi})^{1/\zeta}}^{-1}CA_{j}x_{w}\, d\lambda 
\\ & = \frac{1}{2\pi i}\int_{\Gamma}\bigl(\lambda e^{i\phi}\bigr)^{l-1-w}\sum \limits_{j\in {\mathbb N}_{n}^{0} \setminus S_{\omega}}\bigl(\lambda e^{i\phi}\bigr)^{q_{j}}P_{(\lambda e^{i\phi})^{1/\zeta}}^{-1}CA_{j}x_{w}\, d\lambda =0.
\end{align*}
Taking into account the last formula on p. 170 of \cite{x263}, it readily follows that (\ref{AJ}) golds good.
The proof of theorem is thereby complete in the case that $0<\zeta \leq 1.$ Suppose now $1<\zeta \leq 2.$ Then it is not difficult to show that there exists a finite constant $d_{\zeta}>0$ such that the function $F_{\lambda}(z):=zE_{\zeta,2}(z^{\zeta}\lambda e^{i\phi}),$ $z\in {\mathbb C}$ ($\lambda \in {\mathbb C}$) satisfies 
$F_{\lambda}^{\prime}(z)=E_{\zeta} (z^{\zeta}\lambda e^{i\phi}),$ $z\in {\mathbb C} \setminus (-\infty,0]$ 
and $|F_{\lambda}(z)|\leq d_{\zeta}(1+|z|)e^{|z||\lambda|^{1/\zeta}},$ $z\in {\mathbb C}$ ($\lambda \in {\mathbb C}$). Since for any function $u\in C^{1}([0,\infty) :X)$ with $u^{\prime}(0)=0$ we have ${\mathbf D}_{t}^{\zeta}(g_{1}\ast u)(t)=(g_{1}\ast {\mathbf D}_{\cdot}^{\zeta}u)(t),$ $t\geq 0,$ provided in addition that the term $  {\mathbf D}_{t}^{\zeta}u(t)$ is defined for $t\geq 0,$ it readily follows that
${\mathbf D}_{t}^{\zeta}F_{\lambda}(t)=(g_{1}\ast {\mathbf D_{t}^{\zeta}}E_{\zeta}(\cdot^{\zeta}\lambda e^{i\phi}))(t)=\lambda e^{i\phi}(g_{1}\ast E_{\zeta}(\cdot^{\zeta}\lambda e^{i\phi}))(t)=\lambda e^{i\phi}F_{\lambda}(t),$ $t\geq 0$ ($\lambda \in {\mathbb C}$). 
Let $x_{w},\ y_{w}\in \bigcap_{j=0}^{n}D(A_{j})$ for all $w\in {\mathbb N}_{q_{n}-1}^{0}.$ Define now the solution $u_{\epsilon}(\cdot)$ by 
\begin{align*}
& u_{\epsilon}(z):=\sum \limits_{w=0}^{q_{n}-1}\sum \limits_{j\in S_{\omega}}\frac{1}{2\pi i}\int_{\Gamma}e^{-\epsilon (a-\lambda)^{b/\zeta}}E_{\zeta}\bigl( z^{\zeta}\lambda e^{i\phi}\bigr) \bigl(\lambda e^{i\phi}\bigr)^{q_{j}-1-w}P_{(\lambda e^{i\phi})^{1/\zeta}}^{-1}CA_{j}x_{w}\, d\lambda 
\\ & + \sum \limits_{w=0}^{q_{n}-1}\sum \limits_{j\in S_{\omega}} \frac{1}{2\pi i}\int_{\Gamma}e^{-\epsilon (a-\lambda)^{b/\zeta}}F_{\lambda}(z) \bigl(\lambda e^{i\phi}\bigr)^{q_{j}-1-w}P_{(\lambda e^{i\phi})^{1/\zeta}}^{-1}CA_{j}y_{w}\, d\lambda,
\end{align*}
for any $z\in {\mathbb C}$ and $\epsilon>0.$
It is not difficult to prove that, for every $p\in {\mathbb N}_{0},$ $t\geq 0$ and $\epsilon>0,$ the following holds:
\begin{align*}
& \bigl({\mathbf D}_{t}^{\zeta}\bigr)^{p}u_{\epsilon}(t)
\\ & =\sum \limits_{w=0}^{q_{n}-1}\sum \limits_{j\in S_{\omega}}\frac{1}{2\pi i}\int_{\Gamma}e^{-\epsilon (a-\lambda)^{b/\zeta}}E_{\zeta}\bigl( t^{\zeta}\lambda e^{i\phi}\bigr) \bigl(\lambda e^{i\phi}\bigr)^{p+q_{j}-1-w}P_{(\lambda e^{i\phi})^{1/\zeta}}^{-1}CA_{j}x_{w}\, d\lambda 
\\ & + \sum \limits_{w=0}^{q_{n}-1}\sum \limits_{j\in S_{\omega}} \frac{1}{2\pi i}\int_{\Gamma}e^{-\epsilon (a-\lambda)^{b/\zeta}}F_{\lambda}(t) \bigl(\lambda e^{i\phi}\bigr)^{p+q_{j}-1-w}P_{(\lambda e^{i\phi})^{1/\zeta}}^{-1}CA_{j}y_{w}\, d\lambda
\end{align*}
and
\begin{align*}
& \frac{d}{dt}\bigl({\mathbf D}_{t}^{\zeta}\bigr)^{p}u_{\epsilon}(t)
\\ &=\sum \limits_{w=0}^{q_{n}-1}\frac{1}{2\pi i}\int_{\Gamma}e^{-\epsilon (a-\lambda)^{b/\zeta}}t^{\zeta-1}E_{\zeta,\zeta}\bigl( t^{\zeta}\lambda e^{i\phi}\bigr) \bigl(\lambda e^{i\phi}\bigr)^{p+q_{j}-w}P_{(\lambda e^{i\phi})^{1/\zeta}}^{-1}CA_{j}x_{w}\, d\lambda 
\\ & + \sum \limits_{w=0}^{q_{n}-1}\frac{1}{2\pi i}\int_{\Gamma}e^{-\epsilon (a-\lambda)^{b/\zeta}}E_{\zeta}\bigl( t^{\zeta}\lambda e^{i\phi}\bigr) \bigl(\lambda e^{i\phi}\bigr)^{p+q_{j}-w-1}P_{(\lambda e^{i\phi})^{1/\zeta}}^{-1}CA_{j}y_{w}\, d\lambda .
\end{align*}
The remaining part of proof of theorem can be deduced by repeating almost literally
the arguments that we have already used in the case that $0<\zeta \leq 1.$
\end{proof}

\begin{rem}\label{unexp}
\begin{itemize}
\item[(i)] Theorem \ref{GZA} seems to be new and not considered elsewhere provided that $B\neq I$ or $\zeta \neq 1.$
\item[(ii)] As explained in \cite[Remark 1(iv)]{vlad-mar-prim}, the operator $({\mathbf D}_{s}^{\zeta})^{p}$ and the operator ${\mathbf D}_{s}^{\zeta p}$ can be completely different provided that $\zeta \in (0,2) \setminus \{1\}$ and $p\in {\mathbb N} \setminus \{1\},$ which clearly implies that we have to make a strict distinction between them. Observe also that Theorem \ref{GZA} can be reformulated in the case that $\zeta>2$ (cf. also \cite[Theorem 2.1, Theorem 2.2]{fcaa}) and that we can prove a similar result on the existence and uniqueness of entire and analytical solutions of problem (DFP)$_{R}$ considered in \cite{vlad-mar-prim}; we leave the reader to make this precise. 
\item[(iii)] The notion of an entire solution of the abstract Cauchy problem $(ACP_{n}),$ introduced in \cite[Definition 1.1]{XL-entire}, is slightly different from the corresponding notion introduced in Definition \ref{RES}(ii). Strictly speaking, if $u(\cdot)$ is an entire solution of the abstract Cauchy problem $(ACP_{n})$ in the sense of Definition \ref{RES}(ii), then $u(\cdot)$ is an entire solution of problem $(ACP_{n})$ in the sense of \cite[Definition 1.1]{XL-entire}. The converse statement holds provided that for each index $i\in {\mathbb N}_{n-1}$ the initial values $u_{0},\cdot \cdot \cdot,u_{i-1}$ belong to $D(A_{i}).$  
\item[(iv)] The uniqueness of analytical solutions of problem [(\ref{snarky})-(\ref{puppy})] on the region ${\mathbb C} \setminus (-\infty,0]$ can be proved as follows. Let $u(\cdot)$ be an analytical solution of problem [(\ref{snarky})-(\ref{puppy})] on the region ${\mathbb C} \setminus (-\infty,0],$ with the initial values $u_{j},$ resp. $u_{j},\ v_{j},$ being zeroes ($0\leq j\leq q_{n}-1$). 
Then the choice of initial values in (\ref{puppy}) enables us to integrate the equation (\ref{snarky}) 
$(q_{n}\zeta)$-times by using the formula \cite[(1.21)]{bajlekova}. Keeping in mind the analyticity of $u(\cdot)$, we easily infer that for each $i\in  {\mathbb N}_{n-1}^{0}$ the
mappings $z\mapsto Bu(z),$ $z\in {\mathbb C} \setminus (-\infty,0)$ ($z\mapsto Bu(z),$ $z\in {\mathbb C} \setminus (-\infty,0]$) and $z\mapsto (g_{(q_{n}-q_{i})\zeta} \ast A_{i}u)(z),$ $z\in {\mathbb C} \setminus (-\infty,0)$ ($z\mapsto (g_{(q_{n}-q_{i})\zeta} \ast A_{i}u)(z),$ $z\in {\mathbb C} \setminus (-\infty,0]$) are well defined and continuous (analytical), as well as that
$$
Bu\bigl(te^{i\gamma}\bigr)+\sum \limits_{i=0}^{n-1}\int^{te^{i\gamma}}_{0}g_{(q_{n}-q_{i})\zeta}(s) A_{i}u\bigl( t e^{i\gamma}-s \bigr)\, ds=0,\quad t\geq 0,\ \gamma \in (-\pi,\pi),
$$
i.e., that
\begin{equation}\label{jaco}
Bu\bigl(te^{i\gamma}\bigr)+\sum \limits_{i=0}^{n-1}\bigl(e^{i\gamma}\bigr)^{(q_{n}-q_{i})\zeta}\int^{t}_{0}g_{(q_{n}-q_{i})\zeta}(t-s) A_{i}u\bigl( se^{i\gamma} \bigr)\, ds=0,\quad t\geq 0,\ \gamma \in (-\pi,\pi).
\end{equation}
It is clear that there exists $\gamma \in (-\pi,\pi)$ such that $2re^{-i\gamma \zeta}\in e^{i\phi}\Sigma_{(\zeta \pi/2)+\theta}.$ Setting  $\phi':=\gamma \zeta,$
$u_{\gamma}(t):=u(te^{i\gamma}),$ $t\geq 0,$ we obtain that $u_{\gamma} \in C([0,\infty) : X)$ and
$$
e^{-iq_{n}\phi'}Bu_{\gamma}(t)+\sum \limits^{n-1}_{i=0}e^{-iq_{i}\phi'}A_{i}\bigl(g_{(q_{n}-q_{i})\zeta}\ast u_{\gamma}\bigr)(t)=0,\quad t\geq 0.
$$
On the other hand,
$$
P_{(\lambda e^{-i\phi'})^{1/\zeta}}=\lambda^{q_{n}}e^{-iq_{n}\phi'}B+\sum \limits^{n-1}_{i=0}\lambda^{q_{i}}e^{-iq_{i}\phi'}A_{i},\quad \lambda \in {\mathbb C} \setminus \{0\}.
$$
Using the previous two equalities and \cite[Theorem 2.2]{publi} (applied to the operators $e^{-iq_{i}\phi'}A_{i}$ in place of the operators
$A_{i}$ appearing in the formulation of this theorem), we get that $u_{\gamma}(t)=0,$ $t\geq 0,$ which clearly implies that $u(z)=0,$ $z\in {\mathbb C} \setminus (-\infty,0).$
\end{itemize}
\end{rem}

It is also worth noting that Theorem \ref{GZA} is an extension of \cite[Theorem 2.1]{XL-entire} (cf. also \cite[Theorem 4.2, p. 168]{x263}), where it has been assumed that $B=C=I,$ $\zeta=1,$ $X$ is a Banach space and $\bigcap_{j=0}^{n}D(A_{j})$ is dense in $X$ (in our opinion, the strong continuity in (ii) is very important for the validity of Theorem \ref{GZA} and cannot be so simply neglected here (cf. \cite[(4.8), p. 169]{x263}); also, it ought to be observed that Lemma \ref{tuga-jesenja} is very important for 
filling some absences in the proof of \cite[Theorem 4.2]{x263}, appearing on the lines  1-6, p. 171 in \cite{x263}, where the Cauchy formula has been used by assuming the analyticity of mapping $\lambda \mapsto R_{e^{i\phi}\lambda},$ $\lambda \in \Sigma_{(\pi/2)+\theta},$ $|\lambda|>r$ a priori); observe also that, in the concrete situation of abstract Cauchy problem $(ACP_{n}),$ our estimate on the growth rate of $P_{e^{i\phi}\cdot}^{-1}$ (cf. the equation (\ref{818}) with $\nu'<-(n-1)$) is slightly better than the corresponding estimate \cite[(4.2)]{x263}, where it has been required that $\nu'\leq -n.$
If $B=C=I$ and $\zeta=2,$ then Theorem \ref{GZA} strengthens \cite[Theorem 2.1]{XL-entire} in a drastic manner. Speaking-matter-of-factly, our basic requirement in (i) is that the operator $P_{\lambda}^{-1}=(\lambda^{2n}+\lambda^{2n-2}A_{n-1}+\cdot \cdot \cdot +A_{0})^{-1}$ exists on the region $\{\lambda^{1/2} : \lambda \in  e^{i\phi}\Sigma_{(\zeta \pi/2)+\theta} : |\lambda|\geq r\},$ which can be contained in an arbitrary acute angle at vertex $(0,0);$ on the other hand, in the formulation of \cite[Theorem 2.1]{XL-entire}, T.-J. Xiao and J. Liang require the existence of operator $P_{\lambda}^{-1}$ for any complex number $\lambda$ having the modulus greater than or equal to $r$ and belonging to the obtuse angle 
$e^{i\phi}\Sigma_{(\pi/2)+\theta}$. 

In the following slight modification of \cite[Example 2]{vlad-mar-prim}, we will focus our attention on the possible applications of Theorem \ref{GZA} in which $C$ is not the identity operator on $X;$ 
this is a very important example because of its universality and covering a wide range of concrete applications (here, our attempt is to reconsider and relax, in a certain sense, the very restrictive condition \cite[(4.2), p. 168]{x263}).

\begin{example}\label{not-ultra-entire} 
Suppose that $0< \zeta\leq 2,$ $\theta \in (-\pi,\pi],$ $r>0,$ $q_{n}>q_{n-1},$ $\emptyset \neq \Omega \subseteq {\mathbb C},$ $N\in {\mathbb N},$ $A$ is a densely defined closed linear operator in $X$ satisfying that $\Omega \subseteq \rho(A)$ and the operator family $\{(1+|\lambda|)^{-N}(\lambda-A)^{-1} : \lambda \in \Omega \}\subseteq L(X)$ is equicontinuous (here we can also assume that the operator $A$ is not densely defined or that the operator family $\{(1+|\lambda|)^{-N}(\lambda-A)^{-1}C : \lambda \in \Omega \}\subseteq L(X)$ is equicontinuous, thus providing some applications of Theorem \ref{GZA} to the abstract degenerate fractional equations involving non-elliptic differential operators with the empty resolvent set; see \cite{quan-miao}). Assume, further, that $P_{i}(z)$ is a complex polynomial ($i\in {\mathbb N}_{n}^{0}$), $P_{n}(z)\not \equiv 0,$ 
$\lambda_{0}\in \rho(A) \setminus \{z\in {\mathbb C} : P_{n}(z)=0\},$ dist$(\lambda_{0},\Omega)>0,$
as well as that 
for each $\lambda \in e^{i\phi}\Sigma_{(\zeta \pi/2)+\theta}$ with $|\lambda|\geq r$ all roots of the polynomial
$$
z\mapsto \lambda^{q_{n}}P_{n}(z)+\sum_{i=0}^{n-1}\lambda^{q_{i}}P_{i}(z),\quad z\in {\mathbb C}
$$
belong to $\Omega.$ Set $B:=P_{n}(A)$ and $A_{i}:=P_{i}(A)$ ($i\in {\mathbb N}_{n-1}^{0}$).
Then it can be proved that there exist two sufficiently large numbers $Q'\geq N+2,$ $Q'\in {\mathbb N}$ and $r'>r$ such that the operator families
$\{(1+|\lambda|)^{q_{n}}(\lambda^{q_{n}} P_{n}(A)+\sum_{i=0}^{n-1}\lambda^{q_{i}}P_{i}(A))^{-1}(\lambda_{0}-A)^{-Q'} :\lambda \in e^{i\phi}\Sigma_{(\zeta \pi/2)+\theta},\ |\lambda|\geq r' \}\subseteq L(X)$ and $\{(1+|\lambda|)^{q_{n}}P_{j}(A)(\lambda^{q_{n}} P_{n}(A)+\sum_{i=0}^{n-1}\lambda^{q_{i}}P_{i}(A))^{-1}(\lambda_{0}-A)^{-Q'} :\lambda \in e^{i\phi}\Sigma_{(\zeta \pi/2)+\theta},\ |\lambda|\geq r \}\subseteq L(X)$ are both equicontinuous and strongly continuous ($j\in {\mathbb N}_{n-1}^{0}$).
Hence, Theorem \ref{GZA} can be applied with the regularizing operator $C\equiv (\lambda_{0}-A)^{-Q'}.$ 
\end{example}

In the following theorem, we will reconsider the assertion of \cite[Theorem 2.3.3]{knjigaho} for systems of abstract degenerate differential equations.

\begin{thm}\label{kragujevac}
Let $(X,\|\cdot \|)$ be a complex Banach space and let $iA_{j},\
1\leq j\leq n$ be commuting generators of bounded $C_{0}$-groups on
$X.$ Suppose $\alpha>0,$ $d\in {\mathbb N}$ and $P_{i}(x)= \sum_{|\eta|\leq
d}P_{\eta,i}x^{\eta}$ ($P_{\eta,i}\in {\mathbb C}^{m,m},$ $x\in {\mathbb R}^{n},$ $i=1,2$) are
two given polynomial matrices. Suppose that for each $x\in {\mathbb R}^{n}$ the matrix $P_{2}(x)$ is regular. Then there exists a
dense subset $X_{\alpha,m}$ of $X^{m}$ such that, for every $\vec{x}\in X_{\alpha,m},$ there exists a unique solution (defined in the very obvious way) of the following abstract Cauchy problem: 
\[\hbox{(DFP)}: \left\{
\begin{array}{l}
{\mathbf D}_{t}^{\alpha}\overline{P_{2}(A)}\vec{u}(t)=\overline{P_{2}(A)}{\mathbf D}_{t}^{\alpha}\vec{u}(t)=\overline{P_{1}(A)}\vec{u}(t),\quad t\geq 0,\\
\vec{u}(0)=\vec{x};\quad \vec{u}^{(j)}(0)=0,\ 1\leq j \leq \lceil \alpha \rceil -1.
\end{array}
\right.
\]
Furthermore,  for every $\vec{x}\in X_{\alpha,m},$ the mapping $t\mapsto \vec{u}(t),$ $t\geq 0$ can be extended to the whole complex plane (the extension of this mapping will be denoted by the same symbol in the sequel), and the following holds:
\begin{itemize}
\item[(i)] The mapping $z\mapsto \vec{u}(z),$ $z\in {\mathbb C} \setminus (-\infty,0]$ is analytic.
\item[(ii)] The mapping $z\mapsto \vec{u}(z),$ $z\in {\mathbb C}$ is entire provided that $\alpha \in {\mathbb N}.$
\end{itemize}
\end{thm}

\begin{proof}
Let us recall that $k=1+\lfloor n/2 \rfloor.$ 
Suppose that $P_{1}(x)=[p_{ij;1}(x)]_{1\leq i,j\leq m}$ and $P_{2}(x)=[p_{ij;2}(x)]_{1\leq i,j\leq m}$ ($x\in {\mathbb R}^{n}$), and $d$ is the maximal degree of all non-zero polynomials $p_{ij;1}(x)$
and $p_{ij;2}(x)$ ($1\leq i,j\leq m$).
Then $\sup_{x\in {\mathbb R}^{n}}|\mbox{det}(P_{2}(x))|^{-1}<\infty$ and
we can inductively prove that there exist numbers $M_{1} \geq 1$ and $M_{2}\geq 1$ such that 
for each $l\in {\mathbb N}_{0}$ there exist polynomials $R_{ij;l}(x)$ ($1\leq i,j\leq m$) of degree $\leq lmd$ satisfying that
$$
\Bigl( P_{2}(x)^{-1}P_{1}(x) \Bigr)^{l}=\frac{\bigl[ R_{ij;l}(x) \bigr]_{1\leq i,j\leq m}}{\bigl(\mbox{det}(P_{2}(x))\bigr)^{l}},\quad x\in {\mathbb R}^{n}
$$
and that the following holds:
\begin{align}
\notag \Biggl|  D^{\eta}\Biggl( & \frac{R_{ij;l}(x)}{\bigl(\mbox{det}(P_{2}(x))\bigr)^{l}}\Biggr) \Biggr| +
\Biggl|  D^{\eta}\Biggl( p_{ij;1}(x) \frac{R_{ij;l}(x)}{\bigl(\mbox{det}(P_{2}(x))\bigr)^{l}}\Biggr) \Biggr| 
\\\label{larry} &+
\Biggl|  D^{\eta}\Biggl(p_{ij;2}(x) \frac{R_{ij;l}(x)}{\bigl(\mbox{det}(P_{2}(x))\bigr)^{l}}\Biggr) \Biggr| 
\leq M_{1}^{l}(1+|x|)^{lmdM_{2}},
\end{align}
provided $l\in {\mathbb N}_{0},\ x\in {\mathbb R}^{n},\ 0\leq |\eta|\leq k$ and $ 1\leq i,j\leq m.$
It is very simple to prove that there exists a sufficiently large natural number $k'$ satisfying $2 | k'$ and
\begin{equation}\label{entire-even}
\lim \limits_{l\rightarrow +\infty}\frac{\Bigl(\Gamma\bigl(\frac{2M_{2}(l+1)md+n}{k'd}\bigr)\Bigr)^{1/2l}}{\bigl(\Gamma(\alpha l+1)\bigr)^{1/l}}=0.
\end{equation}
Let $a>0$ be fixed. Set
$C:=(e^{-a|x|^{k'd}})(A)$ and $X_{\alpha,m}:=(R(C))^{m}.$ Then $C\in L(X),$ $C$ is injective and
$D_{\infty}(A_{1}^{2}+\cdot \cdot \cdot +A_{n}^{2})\supseteq R(C) $
is dense in $X$ (\cite{l1}). Define
\begin{equation}\label{WR}
W_{\alpha}(z):=\Biggl[\sum \limits_{l=0}^{\infty}\frac{z^{\alpha
l}}{\Gamma(\alpha l+1)} \Biggl(\frac{R_{ij;l}(x)}{\bigl(\mbox{det}(P_{2}(x))\bigr)^{l}}e^{-a|x|^{k'd}}\Biggr)(A)\Biggr]_{1\leq i,j\leq m},\ z\in {\mathbb C}.
\end{equation}
Using (\ref{larry})-(\ref{entire-even}) and the proof of \cite[Theorem 2.3.3]{knjigaho}, it readily follows that $W_{\alpha}(z)\in L(X^{m})$ for all $z\in {\mathbb C}$, as well as that the expressions
$$
\Biggl[\sum \limits_{l=0}^{\infty}\sum_{v=1}^{m}\frac{z^{\alpha
l}}{\Gamma(\alpha l+1)} \Biggl(p_{iv;2}(x)\frac{R_{vj;l+1}(x)}{\bigl(\mbox{det}(P_{2}(x))\bigr)^{l+1}}e^{-a|x|^{k'd}}\Biggr)(A)\Biggr]_{1\leq i,j\leq m}
$$
and
$$
\Biggl[\sum \limits_{l=0}^{\infty}\sum_{v=1}^{m}\frac{z^{\alpha
l}}{\Gamma(\alpha l+1)} \Biggl(p_{iv;1}(x)\frac{R_{vj;l}(x)}{\bigl(\mbox{det}(P_{2}(x))\bigr)^{l}}e^{-a|x|^{k'd}}\Biggr)(A)\Biggr]_{1\leq i,j\leq m}
$$
define the bounded linear operators on $X^{m}$ ($z\in {\mathbb C}$). Furthermore,
the mapping $z \mapsto W_{\alpha}(z),$ $z\in {\mathbb C} \setminus (-\infty,0]$ is analytic, and
the mapping $z\mapsto W_{\alpha}(z),$ $z\in {\mathbb C}$ is entire provided that $\alpha \in {\mathbb N}.$
Suppose now $\vec{x}\in X_{\alpha,m}.$ Then there exists $\vec{x'}\in X^{m}$ such that $\vec{x}=C_{m}\vec{x'},$ where $C_{m}=CI_{m}.$ Setting $\vec{u}(z):=W_{\alpha}(z)\vec{x'},$ $z\in {\mathbb C},$ we immediately obtain that  
(i) and (ii) hold. It is not difficult to prove that ${\mathbf D}_{t}^{\alpha}(t)\vec{u}(t)$ is well-defined, as well as that
$$
{\mathbf D}_{t}^{\alpha}(t)\vec{u}(t)=\Biggl[\sum \limits_{l=1}^{\infty}\frac{t^{\alpha (
l-1)}}{\Gamma(\alpha (l-1)+1)} \Biggl(\frac{R_{ij;l}(x)}{\bigl(\mbox{det}(P_{2}(x))\bigr)^{l}}e^{-a|x|^{k'd}}\Biggr)(A)\Biggr]_{1\leq i,j\leq m}\vec{x'},\quad t\geq 0,
$$
and $\vec{u}(0)=\vec{x},\ \vec{u}^{(j)}(0)=0,\ 1\leq j \leq \lceil \alpha \rceil -1.$
Since $\overline{P_{1}(A)}$
and $\overline{P_{2}(A)}$ are closed, we can prove with the help of ($\triangleright$) 
that $\vec{u}(t) \in D(\overline{P_{1}(A)}) \cap D(\overline{P_{2}(A)}),$
$
{\mathbf D}_{t}^{\alpha}\vec{u}(t) \in D(\overline{P_{2}(A)}),$ the term ${\mathbf D}_{t}^{\alpha}\overline{P_{2}(A)}\vec{u}(t)$ is well defined, with
\begin{align*}
& \overline{P_{2}(A)}{\mathbf D}_{t}^{\alpha}\vec{u}(t)={\mathbf D}_{t}^{\alpha}\overline{P_{2}(A)}\vec{u}(t)
\\ & =\Biggl[\sum \limits_{l=0}^{\infty}\sum_{v=1}^{m}\frac{z^{\alpha
l}}{\Gamma(\alpha l+1)} \Biggl(p_{iv;2}(x)\frac{R_{iv;l+1}(x)}{\bigl(\mbox{det}(P_{2}(x))\bigr)^{l+1}}e^{-a|x|^{k'd}}\Biggr)(A)\Biggr]_{1\leq i,j\leq m}\vec{x'}
\end{align*}
and
$$
\overline{P_{1}(A)}\vec{u}(t)=\Biggl[\sum \limits_{l=0}^{\infty}\sum_{v=1}^{m}\frac{z^{\alpha
l}}{\Gamma(\alpha l+1)} \Biggl(p_{iv;1}(x)\frac{R_{iv;l}(x)}{\bigl(\mbox{det}(P_{2}(x))\bigr)^{l}}e^{-a|x|^{k'd}}\Biggr)(A)\Biggr]_{1\leq i,j\leq m}\vec{x'},
$$
for any $t\geq 0.$ Since 
$$
P_{2}(x)\frac{\bigl[ R_{ij;l+1}(x) \bigr]_{1\leq i,j\leq m}}{\bigl(\mbox{det}(P_{2}(x))\bigr)^{l+1}}=P_{1}(x)\frac{\bigl[ R_{ij;l}(x) \bigr]_{1\leq i,j\leq m}}{\bigl(\mbox{det}(P_{2}(x))\bigr)^{l}},\quad l\in {\mathbb N}_{0},\ x\in {\mathbb R}^{n},
$$
a simple matricial computation shows that the function $t\mapsto \vec{u}(t),$ $t\geq 0$ is a solution of problem (DFP). Now we will prove the uniqueness of solutions of problem (DFP). Let $t\mapsto \vec{u}(t),$ $t\geq 0$ be a solution of (DFP) with $\vec{x}=0.$
Integrating $\alpha$-times (DFP), we get that $\overline{P_{2}(A)}\vec{u}(t)=\int^{t}_{0}g_{\alpha}(t-s)\overline{P_{1}(A)}\vec{u}(s)\, ds,$ $t\geq 0.$ 
Using this equality, as well as the fact that $\overline{P_{2}(A)}W_{\alpha}(t)-\overline{P_{2}(A)}C_{m}=\overline{P_{1}(A)}(g_{\alpha} \ast W_{\alpha}(\cdot))(t)\in L(X^{m}),$ $t\geq 0,$ and the proof of \cite[Proposition 1.1]{prus}, we obtain that $0=(W_{\alpha}\ast 0)(t)=(\overline{P_{2}(A)}C_{m}\ast \vec{u})(t),$ $t\geq 0,$ so that it suffices to prove that the operator $\overline{P_{2}(A)}$ is injective. Suppose that $\overline{P_{2}(A)}\vec{x}=\vec{0}$
for some $\vec{x}\in X^{m}.$ 
By \cite[Lemma 1.1(a)]{zheng-pacific}, we may assume without loss of generality that $\vec{x}\in X_{0}^{m}$ (cf. ($\triangleright$)).   
It is clear that there exist polynomials $q_{ij}(x)$ ($1\leq i,j\leq m$) such that $P_{2}(x)^{-1}=(\mbox{det}(P_{2}(x)))^{-1}[q_{ij}(x)]_{1\leq i,j\leq m}.$ Using ($\triangleright$), one can prove that $[(\mbox{det}(P_{2}(x)))(A)I_{m}][\phi(A)I_{m}]\vec{x}=[(\phi(x)q_{ij}(x))(A)]_{1\leq i,j\leq m}P_{2}(A)\vec{x}=\vec{0},$ $\phi \in {\mathcal S}({\mathbb R}^{n}).$
By \cite[Remark 4.4(i)]{filomat}, the operator $(\mbox{det}(P_{2}(x)))(A)$ is injective, whence we may conclude that 
$[\phi(A)I_{m}]\vec{x}=\vec{0},$ $\phi \in {\mathcal S}({\mathbb R}^{n}).$ This, in turn, implies $\vec{x}=\vec{0}$ and completes the proof of theorem.
\end{proof}

\begin{rem}\label{denseD}
\begin{itemize}
\item[(i)] 
It can be simply proved that for each $\vec{x}\in X_{\alpha,m}$ the solution $t\mapsto \vec{u}(t),$ $t\geq 0$ possesses some expected properties from Definition \ref{RES}: If $\alpha \notin {\mathbb N},$ then the terms $({\mathbf D}_{t}^{\alpha})^{p}\vec{u}(t),$  $\overline{P_{1}(A)}({\mathbf D}_{t}^{\alpha})^{p}\vec{u}(t)$ and $\overline{P_{2}(A)}({\mathbf D}_{t}^{\alpha})^{p}\vec{u}(t)$ 
are well-defined and can be extended to the whole complex plane, analytically on the region ${\mathbb C} \setminus (-\infty,0]$ and continuously on the region ${\mathbb C} \setminus (-\infty,0);$ if $\alpha \in {\mathbb N},$ then the terms $(d^{\alpha p}/dt^{\alpha p})\vec{u}(t),$  $\overline{P_{1}(A)}(d^{\alpha p}/dt^{\alpha p})\vec{u}(t)$ and $\overline{P_{2}(A)}(d^{\alpha p}/dt^{\alpha p})\vec{u}(t)$ 
are well-defined and can be entirely extended to the whole complex plane ($p\in {\mathbb N}_{0}$).
The assertion of \cite[Theorem 2.3.5]{knjigaho} can be reformulated in degenerate case, as well. 
\item[(ii)]
If $m=1,$ $P_{1}(x)=\sum_{|\alpha|\leq d}a_{\alpha}x^{\alpha},$ $P_{2}(x)=\sum_{|\alpha|\leq d}b_{\alpha}x^{\alpha},$
$x\in {{\mathbb R}^{n}}$ ($a_{\alpha},\ b_{\alpha} \in {\mathbb C}$), 
 $P_{2}(x)\neq 0,$ $x\in {{\mathbb R}^{n}}$ and
$X$ is a
function space on which translations are uniformly bounded and
strongly continuous (for example, $L^{p}({\mathbb R}^{n})$ with
$p\in [1,\infty),$ $C_{0}({\mathbb R}^{n})$ or $BUC({\mathbb
R}^{n});$ notice also that $X$ can be consisted of functions defined
on some bounded domain \cite{l1}, \cite{quan-miao}), then
the obvious choice for $A_{j}$ is $i\partial/\partial x_{j}$ ($1\leq
j\leq n$). In this case, $\overline{P_{1}(A)}$ and $\overline{P_{2}(A)}$ are just the operators
$\sum_{|\alpha|\leq d}a_{\alpha}i^{|\alpha|}(\partial/\partial
x)^{\alpha}$ and $\sum_{|\alpha|\leq d}b_{\alpha}i^{|\alpha|}(\partial/\partial
x)^{\alpha},$ respectively, acting with their maximal distributional domains. Making use of Theorem
\ref{kragujevac} and a slight modification of the formula appearing on l. 1, p. 49 of \cite{knjigaho}, we can prove that for each $\alpha>0$ there exists a
dense subset $X_{\alpha,1}$ of $L^{p}({\mathbb R}^{n})$ such that
the abstract Cauchy problem:
\begin{align*}
\sum_{|\alpha|\leq
d}b_{\alpha}i^{|\alpha|} (\partial/\partial x)^{\alpha}{\mathbf D}_{t}^{\alpha}u(t,x) &=
{\mathbf D}_{t}^{\alpha}\sum_{|\alpha|\leq
d}b_{\alpha}i^{|\alpha|}(\partial/\partial x)^{\alpha}u(t,x)
\\ & =\sum_{|\alpha|\leq
d}a_{\alpha}i^{|\alpha|}(\partial/\partial x)^{\alpha}u(t,x),\
t>0,\ x\in {{\mathbb R}^{n}};
\end{align*}
$$
\frac{\partial^{l}}{\partial t^{l}}u(t,x)_{|t=0}=f_{l}(x),\ \ x\in {{\mathbb R}^{n}},\ l=0,1,\cdot \cdot \cdot, \lceil
\alpha \rceil -1,
$$
has a unique solution (obeying the properties clarified in the part (i) of this remark) provided $f_{l}(\cdot) \in X_{\alpha,1},$
$l=0,1,\cdot \cdot \cdot, \lceil \alpha \rceil -1.$ A similar
assertion can be formulated in $X_{l}$-type
spaces (\cite{x263}); we can also move to the spaces $L^{\infty}({\mathbb
R}^{n}),$ $C_{b}({\mathbb R}^{n})$ or $C^{\sigma}({\mathbb R}^{n})$
($0<\sigma<1$) by using distributional techniques, but then we cannot expect the density of the corresponding subspace $X_{\alpha,1}.$
\item[(iii)] In \cite{filomat}-\cite{publi}, we have recently considered 
the $C$-wellposedness of the abstract degenerate Cauchy problem
\[\hbox{(DFP)}: \left\{
\begin{array}{l}
{\mathbf D}_{t}^{\alpha}\overline{P_{2}(A)}u(t)=\overline{P_{2}(A)}{\mathbf D}_{t}^{\alpha}u(t)=\overline{P_{1}(A)}u(t),\quad t\geq 0,\\
u(0)=Cx;\quad u^{(j)}(0)=0,\ 1\leq j \leq \lceil \alpha \rceil -1,
\end{array}
\right.
\]
where $0<\alpha \leq 2,$ $P_{1}(x)$ and $P_{2}(x)$ are complex polynomials, $P_{2}(x)\neq 0,$ $x\in {\mathbb R}^{n},$ $iA_{j},\ 1\leq j\leq n$ are commuting
generators of bounded $C_{0}$-groups on a Banach space $X,$ thus continuing the research studies of T.-J. Xiao-J. Liang \cite{XL}-\cite{XL-HIGHER}. Denote $\Omega(\omega)=\{ \lambda^{2}
:  \Re \lambda >\omega\},$ if $\omega
>0,$ and $\Omega(\omega)={\mathbb C} \ \setminus \ (-\infty,-\omega^{2}],$ if $\omega \leq 0.$ The basic assumption in \cite{filomat}-\cite{publi} was that 
$$
\sup_{x\in {\mathbb R}^{n}}\Re \Biggl(\Biggl(\frac{P_{1}(x)}{P_{2}(x)}\Biggr)^{1/\alpha}\Biggr)\leq
\omega,
$$
provided $0<\alpha<2,$ and $P_{1}(x)/P_{2}(x) \notin \Omega(\omega),$ $x\in {{\mathbb R}^{n}},$ provided $\alpha=2.$ Observe that our results from the part (i) of this remark can be applied
in the analysis of problem $\hbox{(DFP)}$ in the general case $\alpha>0,$ and also in the case that $0<\alpha \leq 2$ and the above-mentioned requirements are not satisfied.
\end{itemize}
\end{rem}

\subsection{Final conclusions and remarks}
We feel duty bound to say that Theorem \ref{kragujevac} and the conclusions from the parts of (ii) and (iii) of Remark \ref{denseD} cannot be applied in the analysis of $L^{p}$-wellposedness of a great number of very important degenerate equations of mathematical physics, like (cf. the monograph by G. V. Demidenko-S. V. Uspenskii \cite{dem} for further information):
\begin{itemize}
\item[(a)] (The Rossby wave equation, 1939)
$$
\Delta u_{t}+\beta u_{y}=0 \ \ (n=2), \ \ u(0,x,y)=u_{0}(x,y);
$$
\item[(b)] (The Sobolev equation, 1940)
\begin{align*}
\Delta u_{tt}&+\omega^{2} u_{zz}=0 \ \ (n=3),\\ & u(0,x,y,z)=u_{0}(x,y,z),\ u_{t}(0,x,y,z)=u_{1}(x,y,z),
\end{align*}
here $\omega/2$ is the angular velocity;
\item[(c)] (The internal wave equation in the Boussinesq approximation, 1903)
\begin{align*}
\Delta u_{tt}& +N^{2} \bigl( u_{xx}+u_{yy} \bigr)=0 \ \ (n=3),\\ & u(0,x,y,z)=u_{0}(x,y,z),\ u_{t}(0,x,y,z)=u_{1}(x,y,z);
\end{align*}
\item[(d)] (The gravity-gyroscopic wave equation, cf. \cite{gabov})
\begin{align*}
\Delta u_{tt}& +N^{2}\bigl( u_{xx}+u_{yy} \bigr) +\omega^{2}u_{zz}=0 \ \ (n=3),\\ & u(0,x,y,z)=u_{0}(x,y,z),\ u_{t}(0,x,y,z)=u_{1}(x,y,z);
\end{align*}
\item[(e)] (Small amplitude oscillations of a rotating viscous fluid)
\begin{align*}
\Delta u_{tt}& -2\nu \Delta^{2}u_{t}+v^{2}\Delta^{3}u+\omega^{2}u_{zz}=0 \ \ (n=3),\\ & u(0,x,y,z)=u_{0}(x,y,z),\ u_{t}(0,x,y,z)=u_{1}(x,y,z).
\end{align*}
Here  $\omega/2$ is the angular velocity and $\nu >0$ is the viscosity coefficient.
\end{itemize}

Before including some details on the existence and uniqueness of entire solutions of the equations (a)-(e) in $L^{p}$ spaces, we need to explain how one can reformulate
the assertion of Theorem \ref{kragujevac} in the case that there exist a vector $x_{0}\in {\mathbb R}^{n}$ and a non-empty compact subset $K$ of ${\mathbb R}^{n}$ 
such that the matrix $P_{2}(x_{0})$ is singular and $\{ x\in {\mathbb R}^{n} : \mbox{det}(P_{2}(x))=0  \}\subseteq K;$ the analysis of some fractional analogons of (a)-(e) can be carry out similarly and is therefore omitted. Denote by ${\mathcal A}$ the class consisting of those $C^{\infty}({\mathbb R}^{n})$-functions $\phi(\cdot)$ satisfying that there exist two open relatively compact neighborhoods $\Omega$ and $\Omega'$ of $K$ in ${\mathbb R}^{n}$ such that $\phi(x)=0$ for all $x\in \Omega$ and $\phi(x)=1$ for all $x \in {\mathbb R}^{n} \setminus \Omega'.$ Since the estimate (\ref{larry}) holds for all $x \in {\mathbb R}^{n} \setminus \Omega , $ we can define for each $z\in {\mathbb C}$ the matricial operator $W_{\alpha}(z)$ (cf. the proof of Theorem \ref{kragujevac}) by replacing the function $e^{-a|x|^{k'd}}$ in (\ref{WR}) with the function
$\phi(x)e^{-a|x|^{k'd}}.$ Setting $C_{\phi}:=(\phi(x) e^{-a|x|^{k'd}})(A)$ for $\phi \in {\mathcal A}$ (then we do not know any longer whether the set $\bigcup_{\phi \in {\mathcal A}}R(C_{\phi})$ is dense in $X,$ and we cannot clarify whether the operator $C_{\phi}$ is injective or not) and $X_{\alpha,m}':=(\bigcup_{\phi \in {\mathcal A}}R(C_{\phi}))^{m},$ and assuming additionally
the injectivity of matricial operator $\overline{P_{2}(A)}$ on $X^{m},$ then for each $\vec{x}\in X_{\alpha,m}'$ there exists a unique solution $t\mapsto \vec{u}(t),$ $t\geq 0$ of the abstract Cauchy problem (DFP), which can be extended to the whole complex plane, and (i)-(ii) from the formulation of Theorem \ref{kragujevac} continues to hold. Rewriting any of the equations (a)-(e) in the matricial form, and using Lemma \ref{lap-inj}, we obtain that there exists a non-trivial subspace $X_{1,1}'$ of $L^{p}({\mathbb R}^{2}),$ resp. $X_{1,2}'$ of $L^{p}({\mathbb R}^{3}) \times L^{p}({\mathbb R}^{3}),$ such that the equation (a), resp. any of the equations (b)-(e), has a unique entire solution provided $u_{0}(x,y)\in X_{1,1}',$ resp. $(u_{0}(x,y,z),u_{1}(x,y,z)) \in X_{1,2}'$ (similar conclusions can be established for the wellposedness of the equations \cite[(5), (6); Section 4]{fala-prim-0} in $L^{p}(\Omega)$, with $\tau >0,$ $\lambda_{2}> 0$ and $\lambda_{3}> 0$ and $\emptyset \neq \Omega \subseteq {\mathbb R}^{n}$ being an open bounded domain with smooth boundary, which are important in the study of fluctuations of thermoelastic plates and non-stationary processes in thermal fields; the uniqueness of entire solutions of the equation \cite[(5)]{fala-prim-0}, resp. \cite[(6)]{fala-prim-0}, in the case that $\lambda_{2}<0,$ resp. $\lambda_{3}<0,$ cannot be proved here by using the injectivity of associated polynomial matrix operator $\overline{P_{2}(A)},$ and we will only refer the reader to the assertions of \cite[Theorem 4]{fala-prim-0} and \cite[Theorem 2.2]{publi} for further information in this direction). It should be finally noted that we do not know, in the present situation, whether 
the subspace  $X_{1,1}',$ resp. $X_{1,2}',$ of initial values $\vec{x}=u_{0},$ resp. $\vec{x}=(u_{0},u_{1}),$ for which there exists a unique entire solution $t\mapsto \vec{u}(t),$ $t\geq 0$ of the equation (a), resp. any of the equations (b)-(e), can be chosen to be dense in $L^{p}({\mathbb R}^{2}),$ resp. $L^{p}({\mathbb R}^{3}) \times L^{p}({\mathbb R}^{3}).$

{\begin{center}
{\sc ACKNOWLEDGEMENTS}
\end{center}

\ \ \ This research was supported in part by grant 174024 of Ministry
of Science and Technological Development, Republic of Serbia.

\end{document}